\documentclass{amsproc}

\usepackage{amssymb}

\usepackage[cmtip,all]{xy}

\usepackage{amsmath, stackrel, MnSymbol, amscd}

\newtheorem{theorem}{Theorem}[section]
\newtheorem{lemma}[theorem]{Lemma}
\newtheorem{proposition}[theorem]{Proposiiton}

\theoremstyle{definition}
\newtheorem{definition}[theorem]{Definition}

\theoremstyle{remark}
\newtheorem{remark}[theorem]{Remark}

\numberwithin{equation}{section}
\newcommand{\Cal}[1]{{\mathcal #1}}
\DeclareMathOperator{\Spec}{Spec}
\DeclareMathOperator{\op}{op}
\DeclareMathOperator{\Rad}{Rad}
\newcommand{\Z}{\mathbb{Z}}
\newcommand{\rann}{\operatorname{r.ann}}
\newcommand{\lann}{\operatorname{l.ann}}
\DeclareMathOperator{\Aut}{Aut}
\newcommand{\Grp}{\mathsf{Grp}}
\newcommand{\End}{\operatorname{End}}
\newcommand{\Rng}{\mathsf{Rng}}
\newcommand{\Ring}{\mathsf{Ring}}
\newcommand{\Mod}{\operatorname{Mod-\!}}
\newcommand{\lSdp}{\mbox{\rm -}{\mathsf{Sdp}}}
\newcommand{\id}{\mbox{\rm id}}
\newcommand{\soc}{\mbox{\rm soc}}

\begin{document}

  \title[From the point of view of complete multiplicative lattices]{Algebraic structures from the point of view of complete multiplicative lattices}
 
    \author{Alberto Facchini}
\address{Dipartimento di Matematica ``Tullio Levi-Civita'', Universit\`a di 
Padova, 35121 Pa\-do\-va, Italy}
 \email{facchini@math.unipd.it}
\thanks{Partially supported by Ministero dell'Istruzione, dell'Universit\`a e della Ricerca (Progetto di ricerca di rilevante interesse nazionale ``Categories, Algebras: Ring-Theoretical and Homological Approaches (CARTHA)''), Fondazione Cariverona (Research project ``Reducing complexity in algebra, logic, combinatorics - REDCOM'' within the framework of the programme Ricerca Scientifica di Eccellenza 2018), and the Department of Mathematics ``Tullio Levi-Civita'' of the University of Padua (Research programme DOR1828909 ``Anelli e categorie di moduli'').}

%    The 2020 edition of the Mathematics Subject Classification is
%    the current definitive version.

\date{}

      \begin{abstract} General results on multiplicative lattices found recently by Facchini, Finocchiaro and Janelidze have been studied in the particular case of groups by Facchini, de Giovanni and Trombetti. In this paper we prove that these results hold not only for the multiplicative lattices of all normal subgroups of a group, but also for much more general multiplicative lattices. Therefore they can be applied to other algebraic structures, for instance to braces. \end{abstract}

\maketitle

\section{Introduction}

In the history of Mathematics there have been at least  two very successful attempts in trying to put in order algebraic structures. One has been with the creation of Universal Algebra, developed in particular by  Garrett Birkhoff, and the other has been with the discovery of Category Theory first \cite{EM}, and then the development of Categorical Algebra. Smith \cite{S1976}, Hagemann-Hermann \cite{HH1979} and Freese and McKenzie \cite{FM1987} introduced a generalization to any algebra in a congruence-permutable variety (or congruence-modular variety, respectively) of
the commutator for groups. This immediately opened the possibility of applying concepts like solvability, nilpotence and the center, to very general settings (see \cite[p.~283]{BS}).
 In this paper we follow this direction, stressing how the point of view of multiplicative lattices also allows to reorganize several concepts relative to algebraic structures. Of course, our project is not so ambitious like creating a new point of view comparable to the extremely powerful introduction of Universal Algebra and Category Theory! 

\begin{definition}\label{defi1}\emph{
    A multiplicative lattice is a complete lattice $(L,\vee,\wedge)$ equipped with a further binary operation $\cdot{\,}\colon L\times L\to L$ (multiplication) satisfying $x\cdot y\leqslant x\wedge y$ for all $x,y\in L$.} 
\end{definition}

As usual, we will often write $xy$ instead of $x\cdot y$. The smallest and the largest elements of the complete lattice $L$ will be denoted by $0$ and $1$, respectively. Notice that $1$ is not an identity for the multiplication $\cdot$, in general.

Multiplicative lattices are an algebraic structure to which little attention has been devoted, but which already appeared in Krull \cite{[K1924]}, and has been studied by M.~Ward \cite{[W1937], [W1938]}, M.~Ward and R.~P.~Dilworth \cite{[WD1939]}, D.~D.~Anderson \cite{D.D. Anderson, 1A}, E.~W.~Johnson, and J.~A.~Johnson \cite{J, JJ}.
In all these papers, further axioms are required: associativity or commutativity of multiplication, distributivity of multiplication with respect to the join $\vee$, existence of an identity, compatibility of multiplication and partial order, the multiplication is the meet $\wedge$, and so on. We don't require any of these further conditions.

\medskip

It is better to view a multiplicative lattice $L$ not as a complete lattices like in Definition~\ref{defi1}, but as a complete join-semilattice $(L,\vee)$ necessarily with $0$ and $1$. Let's be more precise. A {\em join-semilattice} is a partially ordered set that has a join (a least upper bound) for any nonempty finite subset. Equivalently, a join-semilattice is an  algebraic structure $(L,\vee)$ for which $\vee$ is a binary operation that is associative, commutative and in which every element of $L$ is idempotent.
A {\em complete} join-semilattice is a partially ordered set (define $x\le y$ if and only if $x\vee y=y$) that has a join (a least upper bound) for any subset. In a complete join-semilattice a meet (a greatest lower bound) also exists for any subset, so that complete join-semilattices are exactly complete lattices. Every complete join-semilattice has a least element $0$ (the least upper bound of the empty subset) and a greatest element $1$ (the least upper bound of the improper subset). If we view a complete join-semilattice as an algebraic structure $(L,\vee, 0,1)$, then $(L,\vee, 0)$ is a commutative monoid in which every element is idempotent and in which the identity for the operation $\vee$ is $0$, and $(L,\vee, 0,1)$ is a commutative monoid with zero in which every element is idempotent, the identity for the operation $\vee$ is $0$, and the zero is $1$ (because $x\vee 1=1$ for every $x\in L$). This argument could seem annoying, but it is necessary in order to explain the notion of morphism of multiplicative lattices:  there is no difference between complete lattices and complete join-semilattices, but the difference is in morphisms: a morphism $f\colon L\to M$ of complete join-semilattices is a mapping such that $f(\bigvee X)=\bigvee f(X)$ for every subset $X$ of $L$.

It is better not to look at multiplicative lattices as complete lattices $(L,\vee,\wedge)$ with a further third operation $\cdot\,$: it is better to view them as complete join-semilattices $(L,\vee)$ with a further second operation $\cdot$ for which $x\cdot y\le x$ and $x\cdot y\le y$ for every $x,y\in L$. In some sense, the operation $\cdot$ in multiplicative lattices is a replacement of the operation $\wedge$. That is, in some sense, multiplicative lattices are deformations of lattices: given a lattice $(L,\vee,\wedge)$, the operation $\vee$ remains, but the operation $\wedge$ is replaced by the multiplication $\cdot\,$, and $xy\le x$, $xy\le y$ for every $x,y\in L$.

\begin{lemma}\label{1.1} If $f\colon L\to M$ is a morphism of complete join-semilattices, then there is a unique morphism $u\colon M\to L$ of complete meet-semilattices such that, for every $x\in L$ and every $y\in M$, $f(x)\le y\Leftrightarrow x\le u(y)$. \end{lemma}

Considering the complete lattices $L$ and $M$ as categories, $u$ is a right adjoint of $f$. The category $\mathsf{CML}$ of complete multiplicative lattices is  the category whose morphisms are  the  morphisms $f\colon L\to M$ of complete join-semilattices  such that  $f(x)f(x')\leqslant f(xx')$ for all $x,x'\in L$. Thus a morphism of complete multiplicative lattice is a mapping $f\colon L\to M$ such that $f(\bigvee X)=\bigvee f(X)$ for every subset $X$ of $L$, $f(x)f(x')\leqslant f(xx')$ for all $x,x'\in L$, and $f(1)=1$. 

\medskip

This paper strongly depends on the papers \cite{FFJ} and \cite{FGT}.

\section{Prime elements and Zariski spectrum of a multiplicative lattice}\label{2}

The content of this section is taken from \cite{FFJ}. Let $L$ be a multiplicative lattice.
	An element $p$ of $L$ is a {\em prime element} of $L$ if $p\neq1$ and
	\begin{equation*}
	xy\leqslant p\Rightarrow(x\leqslant p\,\,\,\text{or}\,\,\,y\leqslant p)
	\end{equation*}
	for every $x,y\in L$.
	The set $\mathrm{Spec}(L)$ of all prime elements of $L$ is the {\em Zariski spectrum} (or the {\em prime spectrum}) of $L$.

\medskip

Let $\Cal P(\Spec(L))$ be the power set of $\Spec(L)$, that is, the set of all subsets of $\Spec(L)$. We can define two mappings $$V\colon L\to \Cal P(\Spec(L))$$ by \begin{equation*}
V(x)=\{\,p\in\Spec(L)\mid x\leqslant p\,\}
\end{equation*}    
for every $x\in L$, and $$I\colon \Cal P(\Spec(L))\to L,$$ defined by $I(X)=\bigwedge X$ for every subset $X$ of $\Spec (L)$.

If we view $V$ as a mapping $V\colon L\to \Cal P(\Spec(L))^{\op}$, then $V$ turns out to be a morphism of complete join-semilattices (i.e., $V(\bigvee X)=\bigcap_{x\in X} V(x)$ for every subset $X$ of $L$),  and $I\colon \Cal P(\Spec(L))^{\op}\to L$ is the right adjoint of $V$ in the sense of Lemma~\ref{1.1}. The mapping $V\colon L\to \Cal P(\Spec(L))$ has two further important properties:

(1) $V$ transforms the multiplication in $L$ into the union in ${\mathcal{P}}(\Spec(L))$, that is, $V$ is a magma morphism of the magma $(L,\cdot)$ into the magma (the commutative monoid) $(\mathcal{P}(\Spec(L)),\cup)$: $$V(xy)=V(x)\cup V(y)\qquad\mbox{\rm for every }x,y\in L,$$ and $V(1)=\emptyset$. Hence $V$ is a morphism of multiplicative lattices of $(L,\vee,\cdot)$ into the multiplicative lattice $( \Cal P(\Spec(L))^{\op},\cap,\cup)$.

(2) The image $V(L)$ of the mapping $V$ satisfies the axioms for the closed sets of a topology on $\Spec(L)$, called the {\em Zariski topology}, and $\Spec(L)$ with this topology turns out to be a sober topological space \cite[Lemma 2.6]{FFJ}. 

From the fact that $I$ is a right adjoint of $V$,  we get that $V$ and $I$ induce mutually inverse antiisomorphisms of partially ordered sets between the partially ordered subset $\Rad(L):=I(\mathcal{P}(\Spec(L)))$ of $L$ and the partially ordered subset $V(L)$ of $\mathcal{P}(\Spec(L))$. Hence, we have a lattice isomorphism between $\Rad(L)$, the lattice of all elements of $L$ that are intersections of primes, and $\Omega(\Spec(L))$, the lattice of all open subsets of the Zariski space $\Spec(L)$. We will come back to the elements of $L$ that are intersections of primes in Proposition~\ref{char}.

Since $\Omega(\Spec(L))$ is a sublattice of the distributive lattice $\mathcal{P}(\Spec(L))$, it follows that the lattices $\Rad(L)\cong\Omega(\Spec(L))$ are also distributive lattices, which are complete (though joins of infinite subsets in $\Rad(L)$ are not the same as in $L$, and meets of infinite subsets of $\Omega(\Spec(L))$ are not the same as in $\mathcal{P}(\Spec(L))$, but are the interiors of the intersections).

\begin{theorem} {\rm \cite[Theorem 3.2]{FFJ}} 
	Let $f\colon L\to M$ be a morphism in $\mathsf{CML}$ and let $u\colon M\to L$ be its right adjoint. Then:
	\begin{itemize}
		\item [(a)] If $p$ is a prime element of $M$, then $u(p)$ is a prime element of $L$.
		\item [(b)] The map $\Spec(f)\colon {\Spec}(M)\to\Spec(L)$ defined by ${\Spec}(f)(p)=u(p)$ of every $p\in\Spec(M)$ is continuous. Moreover, $(\Spec(f))^{-1}(\mathrm{V}(x))=\mathrm{V}(f(x))$ for every $x\in L$.
		\item [(c)] The assignment $L\mapsto \Spec(L)$ defines a functor $\Spec\colon\mathsf{CML}^{\mathrm{op}}\to\mathsf{STop}$, where $\mathsf{STop}$ is the category of sober topological spaces.
		\item [(d)]{\rm \cite[12.3(e)]{FFJ}} The  functor $\Spec\colon\mathsf{CML}^{\mathrm{op}}\to\mathsf{STop}$ is the right adjoint of the functor $\Omega\colon \mathsf{STop}\to \mathsf{CML}^{\mathrm{op}}$, that assigns to each sober space $X$ the complete lattice $\Omega(X)$ of its open subsets with multiplication the intersection of two open subsets.
	\end{itemize}
\end{theorem}

Our main application of multiplicative lattices is, for an algebraic structure $A$, its lattice of congruences $\Cal L(A)$. But on $\Cal L(A)$ we must fix a multiplication, and the spectrum then depends on the multiplication chosen on $\Cal L(A)$. For example, consider a discrete valuation ring $R$, for instance $\Z_p$. Its lattice of congruences $\Cal L(R)$ is isomorphic to the complete lattice $\Z_{\le 0}\cup\{-\infty\}$. If on $\Cal L(R)$ we fix the usual multiplication of ideals, we get that the spectrum of $\Cal L(R)$ has two points (the zero ideal and the maximal ideal). But if we fix as multiplication on $\Cal L(R)$ intersection of ideals, then all proper ideals are prime, so that the spectrum of $\Cal L(R)$ is infinite. If we fix as multiplication on $\Cal L(R)$ the constant mapping $0$ (i.e., the product of any two ideals of $R$ is always the zero ideal), then the the spectrum of $\Cal L(R)$ is empty. A further possibility is to fix as product of two ideals $I$ and $J$ in $\Cal L(R)$ for an arbitrary (noncommutative) ring $R$ the two-sided ideal generated by all commutators $[i,j]=ij-ji$, where $i\in I$ and $j\in J$.
Hence the spectrum doesn't depend only on the ring $R$ and the lattice $\Cal L(R)$, but also on the multiplication chosen on the lattice of its congruences.

\subsection{Changing the operation}\label{square}

Another thing to which little attention has been devoted in Algebra until now is composition of operations. It has been studied in \cite{Coreano} and \cite{Sergio}. Given two binary operations $\ast$ and $\circ$ on a set $X$, it is possible to define a third binary operation $\square$ setting $x\square y:=(x\ast y)\circ(y\ast x)$.  For example, given an associative algebra $A$ over a field $k$, it is possible to give $A$ a Lie algebra structure taking as Lie bracket  the commutator $[ x , y ] = x y - y x$. More precisely, for a set $X$, it is possible to consider the set $\Cal M(X)$ of all binary operations on $X$, and define a binary operation $\triangleleft$ on $\Cal M(X)$ setting, for every $x,y\in X$, $x(\ast\triangleleft\circ)y:=(x\ast y)\circ(y\ast x)$. In this way, one gets that $\Cal M(X)$ is a monoid with respect to the operation $\triangleleft$, which is non-commutative if $X$ has at least two elements \cite{Coreano, Sergio}. The identity of the monoid $\Cal M(X)$  is the operation $x\ast y:=x$ for every $x,y\in X$.

We can apply this to our multiplicative lattices $(L,\vee,\cdot)$. Given a multiplicative lattice $(L,\vee,\cdot)$, we can define a new binary operation $\square$ on $L$ setting $x\square y:=(x\cdot y)\vee(y\cdot x)$. Then $(L,\vee,\square)$ turns out to be a multiplicative lattice \cite[Example~3.7]{FFJ}. A second example for this operation on $\Cal M(X)$ is given by Jordan product. If $A$ is an associative algebra over a field of characteristic $\ne 2$, one can construct a Jordan algebra $A^+ $ with the same underlying vector space and the multiplication $x\circ y$  (the {\em Jordan product}) defined by $$x\circ y:={\frac {xy+yx}{2}}.$$ The coefficient $\frac{1}{2}$ is convenient here because if any two elements $x,y$ of $A$ commute, then $x\circ y=x\cdot y$, and if $e$ is any two-sided identity in $A$, then $e$ is also a two-sided identity in $A^+$. This is not necessary for a multiplicative lattice $(L,\vee,\cdot)$, where the operation $\vee$ is always idempotent, so that if any two elements $x,y$ of $(L,\cdot)$ commute, then $x\square y=x\cdot y$, and if $e$ is any two-sided identity in $(L,\cdot)$ then $e$ is also a two-sided identity in $(L,\square)$. Notice that, as for Jordan algebras, the operation $\square$ on $L$ is always commutative.

Assume that $L$ satisfies the {\em monotonicity condition}, i.e., that $x\le y$ and $x'\le y'$ imply $xx'\le yy'$ for every $x,y,x',y'\in L$. It can then be proved that an element $p\in L$ is prime in the multiplicative lattice $(L,\vee,\cdot)$ if and only if it is a prime element in the multiplicative lattice $(L,\vee,\square)$, and that the spectra of these two multiplicative lattices are the same topological space \cite[Example~3.7]{FFJ}. This shows that in the study of multiplicative lattices $L$ with the monotonicity condition and their spectra $\Spec(L)$, it will be often possible to assume the operation $\cdot$ commutative.

\section{Some further terminology in multiplicative lattices}

Keeping in mind the example $\Cal N(G)$ of the multiplicative lattice of all normal subgroups of a group $G$, with the commutator of normal subgroups as multiplication, the following terminology turns out to be very natural.

Let $x$ be an element of a multiplicative lattice $L$. 
The {\em lower central series} (or {\em descending central series}) of $x$ is the descending series $$x=x_1\ge x_2\ge x_3\ge\dots,$$ where $x_{n+1}:=x_n\cdot x$ for every $n\ge 1$. If $x_n=0$ for some $n\ge 1$, then $x$ is {\em left nilpotent}. The element $x$ is {\em idempotent} if $x_2=x\cdot x=x$, and {\em abelian} if $x_2=x\cdot x=0$. Similarly, the element $x$ is {\em right nilpotent} if $_nx=0$ for some $n$, where now in the descending series  the elements $_nx$ are defined recursively  by $_{n+1}x:=x\cdot {}_nx$. Clearly, if the multiplication in the multiplicative lattice $L$ is either associative or commutative, then left nilpotency coincides with right nilpotency.

The {\em derived series} of $x$ \cite[Definition~6.1]{FFJ}  is the descending series $$x:=x^{(0)}\ge x^{(1)}\ge x^{(2)}\ge\dots,$$ where $x^{(n+1)}:=x^{(n)}\cdot x^{(n)}$ for every $n\ge 0$. The term $x':=x_2=x\cdot x=x^{(1)}$ is the {\em derived element} of $x$. The element $x$ of $L$ is {\em solvable} if $x^{(n)}=0$ for some integer $n\ge0$. 

If the multiplication on the lattice is associative, then left nilpotency, right nilpotency and solvability of an element $x\in L$ coincide.

For the multiplicative lattice $\Cal N(G)$ of a group $G$ with operation the commutator of two normal subgroups (which is commutative, but not associative), the element $1=G$ of $\Cal N(G)$ is left (=right) nilpotent as an element of the multiplicative lattice in the sense just described if and only if the group $G$ is nilpotent, is a solvable element if and only if $G$ is a solvable group, is an abelian element if and only if $G$ is an abelian group, and is an idempotent element if and only if the group $G$ is perfect. For the multiplicative lattice $\Cal L(R)$ of a ring $R$ with operation the product of two ideals (which is associative, but not commutative in general), the element $1=R$ of $\Cal L(R)$ is left nilpotent (=right nilpotent=solvable) as an element of the multiplicative lattice if and only if the ring $R$ is nilpotent, and is an abelian element if and only if the ring $R$ is an additive abelian group with the zero multiplication.

\medskip

Recall that an element $x$ of a lattice $L$ is {\em meet-irreducible} if $x=y\wedge z$ implies $x=y$ or $x=z$ for every $y,z\in L$. 

\begin{definition}{\rm An element $s$ of a multiplicative lattice $L$ is {\em semiprime} if $x^2\le s$  implies $x\le s$ for every $x\in L$.}\end{definition}

\begin{lemma}\label{3} In a multiplicative lattice $L$, every prime element is meet-irreduc\-ible and semiprime.\end{lemma}

\begin{proof} Clearly, prime elements are semiprime. Let $p\in P$ be a prime element and suppose $p=y\wedge z$. Then $yz\le y\wedge z=p$ implies that $y\le p$ or $z\le p$. But $p=y\wedge z$ implies $p\le y$ and $p\le z$. Therefore either $p=y$ and $p= z$.
\end{proof}

\begin{theorem}
Let $L$ be a multiplicative lattice that satisfies the monotonicity condition. Then all $\Spec(L)=L\setminus\{1\}$ if and only if $L$ is linearly ordered and every element of $L$ is idempotent.\end{theorem}

\begin{proof} Suppose $\Spec(L)=L\setminus\{1\}$. If $L$ is not linearly ordered, there exist $x,y\in L$ such that $x\nleqslant y$ and $y\nleqslant x$. Then $x$ and $y$ are two elements of $L$ with meet $x\wedge y$, and $x>x\wedge y$, $y>x\wedge y$. This shows that $x\wedge y$ is not meet-irreducible, hence is not prime (Lemma~\ref{3}). But $x>x\wedge y$ shows that $x\wedge y\in L\setminus\{1\}$. A contradiction.

Now let $x$ be an element of $L$ that is not idempotent. Then $x^2<x$, so that $x^2$ is not semiprime. In particular $x^2$ is not prime (Lemma~\ref{3}), another contradiction.

Conversely, assume $L$ linearly ordered and that every element of $L$ is idempotent. Let $p$ be any element of $L$, $p\ne 1$. In order to show that $p$ is prime, suppose $x,y\in L$, $x\nleqslant p$ and $y\nleqslant p$. Then $x>p$ and $y>p$. As $L$ is linearly ordered, it follows that $x\wedge y>p$. By the  monotonicity condition, it follows that $xy\ge(x\wedge y)(x\wedge y)=x\wedge y>p$. This proves that $p$ is prime.
\end{proof}

%Quale delle proprieta' sopraddette (abelian, idempotent, % hyperabelian,  semiprime radical,...) 
% non cambia cambiando $\cdot$ con $\square$?

\section{Hyperabelian multiplicative lattices}

{\em In the first two lemmas of this Section we will suppose that $(L,\vee,\cdot)$ is a multiplicative lattice, which is algebraic (every element of $L$ is the join of compact elements) and satisfies the monotonicity condition. Moreover $C(L)$ will denote the set of all compact elements of $L$. }One has that:

\begin{lemma}\label{111} {\rm \cite[Lemma 6.16]{FFJ}} An element $p\ne 1$ in $L$ is prime if and only if \begin{equation*}
	xy\leqslant p\Rightarrow(x\leqslant p\,\,\,\text{or}\,\,\,y\leqslant p)
	\end{equation*}
	for every $x,y\in C(L)$.\end{lemma}
	
	An {\em $m$-system} in $L$ is a nonempty subset $S$ of $C(L)$ such that for every $x,y\in S$ there exists $z\in S$ such that $z\le xy$.
	
\begin{lemma}\label{msystem} An element $p\in L$ is prime if and only if $S_p:=\{\, c\in C(L)\mid c\nleqslant p\,\}$ is an $m$-system in $L$.\end{lemma}

\begin{proof} Clearly, $p=1$ if and only if $S_p=\emptyset$.

Assume $p$ prime, and fix $x,y\in S_p$. Then $x,y\nleqslant p$, so that  $xy\nleqslant p$ because $p$ is prime. Since $L$ is algebraic, there exists a compact element $z\in L$ such that $z\le xy$ and $z\nleqslant p$. Therefore $z\in S_p$.

Conversely, suppose that $S_p$ is an $m$-system. We already know that $p\ne 1$. By Lemma~\ref{111}, in order to show that $p$ is prime we must show that if $x$ and $y$ are compact elements of $L$ with $xy\le p$, then either $x\le p$ or $y\le p$. Assume by contradiction that $x,y\in C(L)$ and $xy\le p$, but $x,y\in S_p$. Since $S_p$ is an $m$-system, there exists $z\in S_p$ such that $z\le xy$. Then $z\le p$, so $z\nin S_p$, a contradiction.
\end{proof}

We will say that  {\em $m$-distributivity} holds in a multiplicative lattice $(L,\vee,\cdot)$ if $$(x\vee y)z=xz\vee yz\quad\mbox{\rm and}\quad x(y\vee z)=xy\vee xz$$ for every $x,y,z\in L$. Notice that when the operation $\cdot$ is the meet $\wedge$, $m$-distributivity holds if and only if the lattice $L$ is a distributive lattice. Hence multiplicative lattices in which $m$-distributivity holds can be viewed as a deformation of bounded distributive lattices.

\begin{lemma}\label{lll} If $m$-distributivity holds in a multiplicative lattice $L$, then the monotonicity condition holds in $L$.\end{lemma}

\begin{proof} Assume that  $m$-distributivity holds in $L$, and suppose $x\le y$, $x'\le y'$, $x,y,x',y'\in L$. Then $x\vee y=y$ and $x'\vee y'=y'$, so $(x\vee y)(x'\vee y')=yy$. By $m$-distributivity, $xx'\vee xy'\vee yx'\vee yy'=yy'$. Thus $xx'\vee xy'\vee yx'\le yy'$. It follows that $xx'\le yy'$, hence the monotonicity condition holds in $L$.
\end{proof}

\begin{remark}{\rm When $\cdot$ is associative and $m$-distributivity holds, the commutative operation $\square$ considered in Section~\ref{square} satisfies the Jordan identity $(x\square y)\square(x\square x)=x\square(y\square(x\square x))$. Jordan algebras are the commutative algebras that satisfy the Jordan identity.}\end{remark}

\begin{proposition}{\rm \cite[Section 5]{FFJ}} Let $L$ be an $m$-distributive lattice in which $1$ is a compact element. The following conditions are equivalent:
\begin{itemize}
		\item [(a)]
$1\cdot 1=1$ in $L$.
\item [(b)]
Every maximal element of $L\setminus\{1\}$ is prime
\item [(c)]
$L$ has enough primes, in the sense that for every $x\neq1\in L$, there exists a prime $p\in L$ with $x\leqslant p$.\end{itemize}\end{proposition}

\begin{lemma} Let $(L,\vee,\cdot)$ be a multiplicative lattice, which is algebraic and satisfies the monotonicity condition. An element $s\in L$ is semiprime if and only if $c^2\le s$ implies $c\le s$
	for every compact element $c$ of $L$.\end{lemma}

\begin{proof} Suppose that $c^2\le s$ implies $ c\le s$
	for every compact element $c\in L$, and let $x\in L$ be an element such that $x\nleqslant s$. Then there exists a compact element $c_0\in L$ such that $c_0\le x$ and $c_0\nleqslant s$. Then $c_0^2\nleqslant s$. By the monotonicity condition, $c_0^2\le x^2$. Hence $x^2\nleqslant s$. This proves that $s$ is semiprime.\end{proof}
	
	{\em For the rest of this Section we will suppose that $(L,\vee,\cdot)$ is a multiplicative lattice, which is algebraic (every element of $L$ is the join of compact elements) and in which $m$-distributivity holds.}

\begin{lemma}\label{3'} Let $L$ be an $m$-distributive multiplicative lattice. Then an element of $L$ is prime if and only if it is a meet-irreducible semiprime element of $L$.\end{lemma}

\begin{proof} One implication has been already seen in Lemma~\ref{3}. For the other, assume that $m$-distributivity holds in $L$, that $p$ is a meet-irreducible semi\-prime element of $L$, that $a,b\in L$, $a\nleqslant p$ and $b\nleqslant p$. To prove that $p$ is prime, we must show that  $ab\nleqslant p$. Now $a\vee p>p$ and $b\vee p>p$. Since $p$ is meet-irreducible, $(a\vee p)\wedge(b\vee p)>p$. As $p$ is semiprime, it follows that $((a\vee p)\wedge(b\vee p))^2 \nleqslant   p$. But $((a\vee p)\wedge(b\vee p))^2=((a\vee p)\wedge(b\vee p))((a\vee p)\wedge(b\vee p))\le(a\vee p)(b\vee p)=ab\vee ap\vee pb\vee p^2\le ab\vee p$. It follows that $ab\nleqslant p$. 
\end{proof}

The next proposition is the analogue, for multiplicative lattices, of the fact that an ideal of a commutative ring $R$ maximal with respect to the property of being disjoint from a multiplicatively closed subset of $R$, is a prime ideal of $R$.

\begin{proposition}\label{max} Let $S$ be an $m$-system in $L$ and $p\in L$ be a maximal element in the set $\{\,y\in L\mid S\cap\{\, c\in C(L)\mid c\le y\,\}=\emptyset\,\}$. Then $p$ is a prime element of $L$.\end{proposition}

\begin{proof} Assume by contradiction that $S$ is an $m$-system, that $p\in L$ is a maximal element in the set $T:=\{\,y\in L\mid S\cap\{\, c\in C(L)\mid c\le y\,\}=\emptyset\,\}$, but that $p$ is not a prime element. By Lemma~\ref{lll} we can apply Lemma~\ref{111}, so that there exist $x,y\in C(L)$ such that $x\nleqslant p$, $y\nleqslant p$, but $xy\le p$. Thus $x\vee p>p$ and $y\vee p>p$. By the maximality of $p$, one has that $S\cap\{\, c\in C(L)\mid c\le x\vee p\,\}\ne \emptyset$ and $S\cap\{\, c\in C(L)\mid c\le y\vee p\,\}\ne \emptyset$. Hence there exist $x',y'\in S$ such that $x'\le x\vee p$ and $y'\le y\vee p$. By the monotonicity condition, $x'y'\le(x\vee p)(y\vee p)$, hence $x'y'\le xy\vee py\vee xp\vee pp$, so $x'y'\le p$. Since $S$ is an $m$-system, there exists $z\in S$ such that $z\le x'y'$. Thus $z\le p$. Hence $z\in S\cap \{\, c\in C(L)\mid c\le p\,\}$, which is therefore not empty, so that $p$ is not an element in the set $T$. This is a contradiction, because  $p$ was a maximal element in $T$.
\end{proof}

\begin{lemma} \label{bjip} Let $I$ and $V$ be the mappings defined at the beginning of Section~\ref{2}, and let $x$ be an element of $L$. If $x<IV(x)$, then $x$ is not semiprime.\end{lemma}

\begin{proof} Let $x$ be an element of $L$. Assume by contradiction that $s:=IV(x)>x$ and $x$ is semiprime, so that $y^2\le x$ implies $y\le x$ for every element $y\in L$. From $x<s$ and the fact that $L$ is an algebraic lattice, it follows that there exists an element $c_0\in C(L)$ such that $c_0\le s$ and $c_0\nleqslant x$. Then $c_0^2\nleqslant x$, so that here exists an element $c_1\in C(L)$ such that $c_1\le c_0^2$ and $c_1\nleqslant x$. In this way it is possible to construct a descending chain $c_0\ge c_1\ge c_2\ge\dots$ of elements of $C(L)$ such that $c_n\nleqslant x$ and $c_{n+1}\le c_n^2$ for every $n\ge 0$.
Set $S:=\{\, c_n\mid n\ge 0\,\}$. The nonempty set $S$ is an $m$-system in $L$, because if $n,m$ are natural numbers and $\ell$ is the greatest between them, then $c_nc_m\ge c_{\ell}c_{\ell}\ge c_{\ell+1}$. Notice that $x$ belongs to the set $T:=\{\,y\in L\mid S\cap\{\, c\in C(L)\mid c\le y\,\}=\emptyset\,\}$. It is possible to apply Zorn's Lemma to the nonempty set $U:=\{\,y\in T\mid x\le y\,\}$, because if $\{\,u_\lambda\mid \lambda\in\Lambda\,\}$ is a nonempty chain in $U$, then $u:=\bigvee_{\lambda\in\Lambda}u_\lambda\in L$ is in $U$, because if $S\cap\{\, c\in C(L)\mid c\le u\,\}\neq\emptyset$, there exists $c\in S$ such that $c\le u=\bigvee_{\lambda\in\Lambda}u_\lambda\in L$. But $c$ is compact, hence there exists $\lambda_0\in\Lambda$ such that $c\le u_{\lambda_0}$. Thus $S\cap\{\, c\in C(L)\mid c\le u_{\lambda_0}\,\}\ne\emptyset$, which is a contradiction. This shows that it is possible to apply Zorn's Lemma to $U$, hence there is a maximal element $p$ in $U$, which is a maximal element in $T$, hence $p$ is prime by Proposition~\ref{max}. Now $p\in U$ implies $x\le p$, so $p\in V(x)$. Therefore $s\le p$. Since $p\in U$, we know that $S\cap\{\, c\in C(L)\mid c\le p\,\}=\emptyset$. But 
$c_0\le s\le p$ and $c_0\in S$. This contradiction proves the Lemma.
\end{proof}

\begin{proposition}\label{char} The following conditions are equivalent for an element $t$ of $L$:

{\rm (a)} $t=\bigwedge P$ for some set $P$ of prime elements of $L$.

{\rm (b)} $t$ is semiprime.\end{proposition}

\begin{proof} Suppose that (a) holds, $x\in L$ and $x^2\le t$. Then $x^2\le p$ for every $p\in P$, so that $x\le p$. It follows that $x\le t$, as desired. This proves that $t$ is semiprime.

Conversely, suppose that (b) holds. Let $P:=V(x)$ be the set of all prime elements $p$ of $L$ such that $t\le p$. By Lemma~\ref{bjip}, we get from (b) that $t:=\bigwedge P$.
\end{proof}

The {\em semiprime radical} of $L$ is the meet of all prime elements of $L$. The multiplicative lattice $L$ is {\em semisimple} if its semiprime radical is $0$. The complete distributive lattice $\Rad(L)$ introduced in Section~\ref{2} consists of all semiprime elements of $L$.

\begin{theorem}\label{hyper} Let $(L,\vee,\cdot)$ be an algebraic multiplicative lattice in which $m$-distributivity holds. The following conditions are equivalent:

{\rm (a)} $1$ is the unique semiprime element of $L$.

{\rm (b)} For every $x\in L$, $x\ne 1$, there exists $y\in L$ such that $y>x$ and $y^2\le x$.

{\rm (c)} There exists a strictly ascending chain $$0:=x_0\le x_1\le x_2\le\dots\le x_{\omega}\le x_{\omega+1}\le \dots\le x_{\alpha}:=1$$ in $L$  indexed in the ordinal numbers less or equal to  $\alpha$ for some ordinal $\alpha$, such that $ x_{\beta+1}^2\le x_{\beta}$ for every ordinal $\beta<\alpha$ and $x_\gamma=\bigvee_{\beta<\gamma}x_\beta$ for every limit ordinal $\gamma\le\alpha$.

{\rm (d)} The lattice $L$ has no prime elements, that is, $\Spec(L)=\emptyset$.

{\rm (e)}  The semiprime radical of $L$ is $1$.

{\rm (f)}  Every $m$-system of $L$ contains $0$.\end{theorem}

\begin{proof} (a)${}\Rightarrow{}$(b) Suppose that (a) holds. Let $x$ be an element of $L$, $x\ne 1$. By (a), $x$ is not semiprime. Hence there exists $z\in L$ such that  $z^2\le x$ and $z\nleqslant x$. Then $y:=z\vee x>x$ is such that $y^2=z^2\vee zx\vee xz\vee x^2\le z^2\vee x=x$. This proves (b).

 (b)${}\Rightarrow{}$(c) Suppose that (b) holds. The constructions of the strictly ascending chain of elements $x_\beta$ of $L$ is by transfinite induction. Set $x_0:=0$. Assume that $x_\beta$ has been defined. If $x_\beta=1$, then $\beta$ is the required ordinal $\alpha$, and the construction stops. If $ x_\beta<1$, by (b) there exists $x_{\beta+1}\in L$ such that $x_{\beta+1}>x_{\beta}$ and $x_{\beta+1}^2\le x_{\beta}$. Finally, if $\gamma$ is a limit ordinal, set $x_\gamma:=\bigvee_{\beta<\gamma}x_\beta$.
 
  (c)${}\Rightarrow{}$(d) Transfinite induction on the ordinal number $\alpha$. If $\alpha=0$, then $L=\{0\}$, and $L$ has no prime elements. 
  
  Notice that, for every $x\in L$, the closed interval $[0,x]:=\{\,y\in L\mid y\le x\,\}$ is a multiplicative sublattice of $L$, because $[0,x]$ is a complete sublattice, which is closed under multiplication because if $y,y'\in[0,x]$, then $yy'\le y\wedge y'\le x\wedge x=x$. (But notice that the greatest element $x$ of $[0,x]$ is in general different from the greatest element $1$ of $L$.)
  
  Suppose $p\in L$ is a prime element of $L$. There are two cases: $\alpha$ limit or $\alpha$ nonlimit.
  
  If $\alpha$ is a limit ordinal, then $1=\bigvee_{\beta<\alpha}x_\beta$. Since $p<1$, there exists $\beta_0<\alpha$ such that $x_{\beta_0}\nleqslant p$. Let $\beta_1<\alpha$ be the smallest ordinal such that $x_{\beta_1}\nleqslant p$. Then $x_\gamma\le p$ for every $\gamma<\beta_1$. Notice that $\beta_1$ must be necessarily a nonlimit ordinal, so $\beta_1=\beta_2+1$ for some ordinal $\beta_2$. Thus $x_{\beta_2}\le p$ and $x_{\beta_2+1}\nleqslant p$. But $x_{\beta_2+1}^2\le x_{\beta_2}\le p$ implies  $x_{\beta_2+1}\le p$, a contradiction. This proves that $\alpha$ must be a nonlimit ordinal, $\alpha=\alpha_0+1$ say. Then $x_{\alpha_0}\nleqslant p$, otherwise $1\cdot 1=x_{\alpha_0+1}\cdot x_{\alpha_0+1}\le x_{\alpha_0}\le p$, hence $1\le p$, contradiction. Thus  $x_{\alpha_0}\wedge p< x_{\alpha_0}$, and the element $x_{\alpha_0}\wedge p$ of the interval $[0, x_{\alpha_0}]$ is prime in $[0, x_{\alpha_0}]$, because if $a,b\in [0, x_{\alpha_0}]$ and $ab\le p$, then either $a\le p$ or $b\le p$, so that either $a\le x_{\alpha_0}\wedge p$ or $b\le x_{\alpha_0}\wedge  p$. This contradicts the inductive hypothesis.
  
   (d)${}\Rightarrow{}$(a) Follows immediately from the characterization of semiprime elements as intersections of primes (Proposition~\ref{char}).
   
    (a)${}\Rightarrow{}$(e) Follows immediately from the definition of semiprime radical as the meet of all prime elements.
    
     (e)${}\Rightarrow{}$(f) Suppose that there exists an $m$-system $S$ in $L$ such that $0\nin S$. We will apply Proposition~\ref{max} showing that the set $T:=\{\,y\in L\mid S\cap\{\, c\in C(L)\mid c\le y\,\}=\emptyset\,\}$ has a maximal element, hence $L$ has a prime element, which contradicts (e). Now $0\in T$ because $0\nin S$. Hence it suffices to apply Zorn's Lemma, that is, it suffices to show that $T$ is inductive. The proof of this is exactly the same as the proof of Lemma~\ref{bjip}.
     
      (f)${}\Rightarrow{}$(d) If $p$ is a prime element of $L$, then by Lemma~\ref{msystem} the set $S_p:=\{\, c\in C(L)\mid c\nleqslant p\,\}$ is an $m$-system in $L$, and $0\nin S_p$.
      \end{proof}

We say that a multiplicative lattice $L$ is {\em hyperabelian} if it satisfies the equivalent conditions of Theorem~\ref{hyper}.

\subsection{Infinite $m$-distributivity} In this subsection we consider {\em a multiplicative lattice $L$ in which $\left(\bigvee X\right)\left(\bigvee Y\right)=\bigvee_{(x,y)\in X\times Y}x y$ for any two subsets $X,Y\subseteq L$ (infinite $m$-distributivity both on the right and on the left).}

For such a multiplicative lattice $L$, it is possible to consider, for any element $x\in L$, the {\em right annihilator} $\rann_L(x)$ of $x$. It is the join of the set $A$ of all elements $y\in L$ such that $xy=0$. Clearly, $\rann_L(x)$ is the greatest element in this set $A$. Similarly for the {\em left annihilator} $\lann_L(x):=\bigvee\{\,y\in L\mid yx=0\,\}$ of $x$. For instance, if $L:=\Cal N(G)$ is the multiplicative lattice of all normal subgroups of a group $G$, then, for every $N\in\Cal N(G)$, $\rann_L(N)=\lann_L(N)$ is the centralizer of $N$ in $L$. It is also possible to define the {\em right center} $rZ(x):=x\wedge \rann_L(x)$ and the
{\em left center} $lZ(x):=x\wedge \lann_L(x)$ of any element $x\in L$. 

The {\em left upper central series} (or {\em left ascending central series}) of the lattice $L$ is the increasing sequence $0:=z_0\le z_1\le\dots\le z_{\omega}\le\dots$ of elements of $L$, indexed in the ordinal numbers, where, for every ordinal $\gamma$, $z_{\gamma+1}$ is the greatest of the elements $z\in L$ such that $z\cdot 1\le z_{\gamma}$. If $\gamma$ is a limit ordinal, then $z_{\gamma}$ is the join of  the set of all the elements $z_{\beta}$ with $\beta\le\gamma$. The {\em left hypercenter} of $L$ is the largest element of the left upper central series. The lattice $L$ is {\em left hypercentral} if the left upper central series reaches~$1$.

\begin{remark}{\rm Notice that if $L$ is a multiplicative lattice and $h\in L$, then the interval $[0,h]_L:=\{\,x\in L\mid x\le h\,\}$ in $L$ is a multiplicative sublattice of $L$. Clearly, $0$ is a prime element of $[0,h]_L$ if and only if for every $a,b\in L,\ a\le h,\ b\le h, ab=0\Rightarrow a=0$ or $b=0$.}\end{remark}

Let us prove that \cite[Lemma 3.1]{FGT} extends from groups to multiplicative lattices with infinite distributivity.

\begin{lemma} Let $L$ be a multiplicative lattice with infinite $m$-distributivity, $h\in L$, and $n\in [0,h]_L$. If $0$ is a prime element of $[0,n]_L$, then $\lann_{[0,h]}(n)$ coincides with $\rann_{[0,h]}(n)$ and is the unique element $p\in\Spec([0,h]_L)$ such that $p\wedge n=0$.\end{lemma}

\begin{proof} Consider the property ``$p\wedge n=0$'' of an element $p$ of $L$.

{\em Step 1: $\rann_{[0,h]}(n)$ has the property.}

In fact, $$(\rann_{[0,h]}(n)\wedge n)^2=(\rann_{[0,h]}(n)\wedge n)(\rann_{[0,h]}(n)\wedge n)\le n\cdot\rann_{[0,h]}(n)=0.$$ Since $0$ is prime in $[0,n]_L$, it follows that $\rann_{[0,h]}(n)\wedge n=0$.

\smallskip

{\em Step 2: $\rann_{[0,h]}(n)$ is prime in $[0,h]_L$}.

First of all notice that $\rann_{[0,h]}(n)<h$, otherwise $nh=0$, so that $n^2=0$, so $n=0$, hence $0$ is not prime in $[0,n]_L$, a contradiction.

Suppose $a,b\in L$, $a,b\le h$, and $ab\le \rann_{[0,h]}(n)$. Then $(a\wedge n)(b\wedge n)\le ab\wedge n\le \rann_{[0,h]}(n)\wedge n=0$ by Step 1. Since $0$ is prime in $[0,n]_L$, it follows that either $a\wedge n=0$ or $b\wedge n=0$. If, for instance,  $a\wedge n=0$, then $na=0$, so  $a\le \rann_{[0,h]}(n)$. Similarly if $b\wedge n=0$. 

\smallskip

{\em Step 3:  $\rann_{[0,h]}(n)$ is the unique prime element of $[0,h]_L$ with the property.}

Let $p$ be a prime element of $[0,h]_L$ with the property, i.e., such that $p\wedge n=0$. Then $np=0$, so $p\le \rann_{[0,h]}(n)$. From $n\cdot \rann_{[0,h]}(n)=0\le p$, it follows that either $n\le p$ or $ \rann_{[0,h]}(n)\le p$. If $n\le p$, then $n=p\wedge n=0$, so  $0$ is not prime in $[0,n]_L$, a contradiction. Thus $ \rann_{[0,h]}(n)\le p$. Since we have already seen that $p\le \rann_{[0,h]}(n)$, we get that $p= \rann_{[0,h]}(n)$, which concludes Step 3.

\smallskip

{\em Step 4:  $\rann_{[0,h]}(n)=\lann_{[0,h]}(n)$.}

We have seen that $\rann_{[0,h]}(n)$ is the unique element $p\in\Spec([0,h]_L)$ such that $p\wedge n=0$. This is a right/left symmetric property, so that $\lann_{[0,h]}(n)$ is also the unique element $p\in\Spec([0,h]_L)$ such that $p\wedge n=0$. This concludes the proof of the Lemma.
 \end{proof}

\section{Some examples.}

\subsection{Groups} We have already mentioned above the example of groups. For any fixed group $G$, consider the complete modular lattice $\Cal N(G)$ of all normal subgroups of $G$, i.e., in the language of Universal Algebra,  the lattice of all congruences of $G$. It is possible to give $\Cal N(G)$ a multiplicative lattice structure in a number of ways. The most interesting one is that of defining as a product of $A,B\in\Cal N(G)$ the commutator $[A,B]$, which is a normal subgroup of $G$ contained in $A$ and $B$. This multiplication on $\Cal N(G)$ is commutative,  but not associative. But it is also possible to define as product of $A,B\in\Cal N(G)$ the intersection $A\cap B$, or the trivial subgroup of $G$ with one element. In this way we get three different structures of multiplicative lattice on $\Cal N(G)$. Of course, it is not possible to take as a product in the multiplicative lattice $\Cal N(G)$ the product $AB$ of two normal subgroups $A,B\in\Cal N(G)$, because it is not true that $AB\subseteq A\cap B$: in fact, $AB$ is the join $A\vee B$ in the complete lattice $\Cal N(G)$.

\bigskip

Notice that groups behave so similarly to commutative rings without identity because in the multiplicative lattice $\Cal N(G)$ of any group $G$, the  multiplication $[-,-]$ of $\Cal N(G)$ is always a commutative operation, though it is not associative. 

\subsection{$G$-groups}
Another immediate application of our theory is to $G$-groups, which we studied in \cite{1,2,3,4,5}. The notion of $G$-group $H$ is classical. Sometimes $G$ is called an {\em operator group} on $H$ \cite[Definition~8.1]{Suzuki}. For any two groups $G$ and $H$ we say that $H$ is a {\em (left) $G$-group} if there is a fixed group homomorphism $\alpha\colon G\to \Aut(H)$. That is, a $G$-group is a pair $(H,\alpha)$, where $H$ is a group and  $\alpha\colon G\to \Aut(H)$ is a group morphism. Equivalently, a $G$-group can be defined as a group $H$ endowed with a mapping $\cdot{\,}\colon   G\times H\to H$, $ (g,h)\mapsto gh$, called {\em left scalar multiplication}, such that 
   
 \noindent {\rm (a)} $g(hh')=(gh)(gh')$,

 \noindent {\rm (b)} $(gg')h = g(g'h)$, and 

  \noindent {\rm (c)} $1_Gh=h$

\noindent for every $g,g'\in G$ and every $h,h'\in H$.

The reader will notice the analogy between the category $G$-$\Grp$ of $G$-groups and the category $R$-Mod of all left modules over a fixed ring $R$. The category $R$-Mod has as objects the pairs $(M,\lambda)$, where $M$ is an additive abelian group and $\lambda\colon R\to\End(M)$ is a ring morphism. In the category $G$-$\Grp$ the morphisms $f\colon(H,\varphi)\to(H',\varphi')$ are the group homomorphisms $f\colon H\to H'$ such that $f(gh)=gf(h)$ for every $g\in G$, $h\in H$, exactly as in the category $R$-Mod. 

We say that $H$ is an {\em abelian $G$-group} if $H$ is a $G$-group and $H$ is an abelian group. Most of the results in \cite{FGT} hold not only for groups, but, more generally, for arbitrary $G$-groups.

The most important example of $G$-group is the regular left $G$-group, that is, the
group $G$ on which $G$ acts via the inner automorphisms of $G$. More precisely, for any fixed group $G$,  let
 $\alpha_g\colon G\to G,\ \alpha_g\colon h\mapsto ghg^{-1}$, be the {\em inner automorphism} of $G$ determined by $g$, for every $g\in G$.
 There is a canonical group morphism $\alpha\colon G\to\Aut(G)$, $g\mapsto\alpha_g$, which makes $G$ a $G$-group. It is called the {\em regular} left $G$-group, similarly to the left $R$-module $_RR$, which is called the regular left $R$-module.
 Other important examples of $G$-groups are the normal subgroups $N$
of $G$ (they are the $G$-subgroups of the regular left $G$-group $G$), and the factor groups $G/N$
(=\,homomorphic images of the regular left $G$-group $G$). Notice that the lattice of all subgroups of $G$ is the lattice of all $\{1\}$-subgroups of the $\{1\}$-group $G$, the lattice of all normal subgroups of $G$ is the lattice of all $G$-subgroups of the regular left $G$-group $G$, the lattice of all characteristic subgroups of $G$ is the lattice of all $\Aut(G)$-subgroups of the $\Aut(G)$-group $G$.

For any fixed $G$-group $H$, the associated multiplicative lattice is the lattice of all normal $G$-subgroups of $H$, i.e., the normal subgroups of $H$ that are closed under the action of the elements of $G$. Multiplication in this lattice is given by the commutator subgroup.

\subsection{Rings without identity} For rngs (associative rings possibly without identity) the associated multiplicative lattice is the same as for rings with identity. Since now we will be dealing with rngs (associative rings without identity), and their theory goes back to the times of Jacobson and Kaplansky, let us briefly recall the theory of Jacobson radical for rings not necessarily with identity. Let $\Rng$ denote the category of all rings, with or without identity, in which morphisms are the mappings that preserve addition and multiplication, and let $\Ring$ be the category of all rings with identity, in which morphisms are the mappings that preserve addition, multiplication and the identities. There is a further third category involved. It is the category $\Ring_a$ of all rings with augmentation, whose objects are the pairs $(S,\varepsilon_S)$, where $S$ is an object in $\Ring$ and $\varepsilon_S\colon S\to\Z$ is a morphism in $\Ring$. Morphisms $f\colon (S,\varepsilon_S)\to (S',\varepsilon_{S'})$ in $\Ring_a$ are the morphisms $f\colon S\to S'$ in $\Ring$ such that $\varepsilon_{S'}f=\varepsilon_S$: $$
\begin{array}{ccccc}
\Z&\longrightarrow&S&\stackrel{\varepsilon_S}{\longrightarrow}&\Z
\\
\Vert && \phantom{\scriptstyle f}\downarrow {\scriptstyle f} &&\Vert \\
\Z&\longrightarrow&S'&\stackrel{\varepsilon_{S'}}{\longrightarrow}&\Z.
\end{array}
$$ There is an equivalence of categories $F\colon\Rng\to\Ring_a$ that associates with any object $R$ of $\Rng$ the abelian group  direct sum $R_+:=R\oplus \Z$ with multiplication defined by
$(r,z)\cdot(r',z')=(rr'+zr'+z'r, zz')$ for every $(r,z),(r',z')\in  R_+.$ This ring $R_+$ is a ring with augmentation. Notice that there is a short exact sequence $0\to R\to R_+\to\Z\to0$ in the category $\Rng$, which is a pointed category with zero object the zero ring with one element. In particular, $R$ is an ideal of the ring with identity $R_+$. The quasi-inverse of the category equivalence $F$ associates with each object $(S,\varepsilon_S)$ of $\Ring_a$ the kernel of $\varepsilon_S$.

The construction of the ring $R_+$ goes back to J. L. Dorroh (\cite{Dorroh}, but also see \cite{Mesyan}). The right (left, two-sided) ideals of $R$ are exactly the right (left, two-sided) ideals of $R_+$ contained in $R\oplus 0$. Conversely, the right (left, two-sided) ideals of $R_+$ are all of the form $K=\{\,(-r,a)\mid r\in J, \ a\in A,\ \varphi(r)=a+Z\,\}$, where $Z\le A$ are ideals of the ring of integers $\Z$, $J$ is a right (left, two-sided, resp.) ideal of $R$, and $\varphi\colon J\to A/Z$ is a morphism in the category $\Rng$.

Given any left module over a rng $R$, that is, a pair $(M,\lambda)$, where $M$ is an additive abelian group and $\lambda\colon R\to\End(M)$ is a morphism in $\Rng$, then $\lambda$ extends uniquely to a morphism $\lambda_+\colon R_+\to\End(M)$ in $\Ring$. This assignement defines a category isomorphism between the categories $\Mod R$ and $\Mod R_+$.  	

\bigskip

A similar situation occurs in several different settings. For instance, let $G$ be a group. Then we can associate with every $G$-group $(H,\varphi)$ the semidirect product $H\rtimes G$, getting an equivalence between the category $G$-$\Grp$ of $G$-groups and the category $G \lSdp$ of all $G$-semidirect products. The objects of $G$-{\bf Sdp} are the triples $(P,\alpha,\beta)$, where $P$ is a group and $\alpha\colon G\to P$, $\beta\colon P\to G$ are morphisms such that the composite mapping $\beta\alpha$ is the identity automorphism $id_G$ of the group $G$. The morphisms from $(P,\alpha,\beta)$ to $(P',\alpha',\beta')$ in $G$-{\bf Sdp} are all group morphisms $a\colon P\to P'$ which make the diagram
$$
\begin{array}{ccccc}
G&\stackrel{\alpha}{\longrightarrow}&P&\stackrel{\beta}{\longrightarrow}&G
\\
\Vert && \phantom{\scriptstyle a}\downarrow {\scriptstyle a} &&\Vert \\
G&\stackrel{\alpha'}{\longrightarrow}&P'&\stackrel{\beta'}{\longrightarrow}&G
\end{array}
$$ commute. Here the vertical arrows on the right and on the left are the identity morphism $id_G$ of $G$. Thus the diagram commutes if and only if
$a\alpha=\alpha'$ and $\beta' a=\beta$. The category $G$-{\bf Sdp} is the category of pointed objects over the $G$-group~$G$ (Borceux \cite[page 28]{Bor}; also see \cite{3}). Let $(H,\varphi )$ be a $G$-group. Let $P$ be the semidirect product $H\rtimes G$, that is, the cartesian product $H\times G$ with the operation defined by $(h_1, g_1 )(h_2, g_2 )=(h_1\varphi (g_1 )(h_2), g_1 g_2)$.
 Let $\alpha\colon G\to P$ and let $\beta\colon P\to G$ be defined by $\alpha (g)= (1,g)$ and   $\beta (h,g)= g$. Then $\alpha $ and $\beta $ are group morphisms and $\beta \alpha = id_G$, so that $(P,\alpha,\beta)$ is an object of $G$-{\bf Sdp}.

If $( H,\varphi )$ and $( H^{'} ,\varphi^{'}  )$ are $G$-groups, $(P,\alpha,\beta)$ and $(P',\alpha',\beta')$ are the corresponding objects in $G$-{\bf Sdp} and $ f \colon(H,\varphi ) \to (H',\varphi')$ is a $G$-group morphism, define $\widetilde{f}\colon P=H\rtimes G\to P'=H'\rtimes G$ by $\widetilde{f}(h,g)=(f(h),g)$ for every $(h,g)\in P=H\rtimes G$. Then $\widetilde{f}$ is a morphism in the category $G$-{\bf Sdp}. The assignments $( H,\varphi )\mapsto (P,\alpha,\beta)$ and $f\mapsto\widetilde{f}$ define a category equivalence $F\colon G$-$\Grp$${}\to G$-{\bf Sdp} (see \cite[Proposition~5.7]{Bor}  or \cite[4.1]{3}).

\bigskip

An even clearer analogy is with the category of $k$-algebras. Let $k$ be a commutative ring with identity. Recall that a $k$-algebra $R$ is a ring that is also a $k$-module $R_k$, and in which $(r s)\lambda= (r\lambda)s = r (s\lambda)$ for every $r,s\in R$ and $\lambda \in k$.   When $R$ is a ring with identity, this is equivalent to the existence of a morphism  $\alpha\colon k\to R$ in $\Ring$ with $\alpha(k)$ contained in the center of $R$. An augmented $k$-algebra is a triple  $(R,\alpha,\beta)$, where $R$ is a ring with identity and $\alpha\colon k\to R$, $\beta\colon R\to k$ are morphisms in $\Ring$ with $\alpha(k)\subseteq Z(R)$ and $\beta\alpha=id_k$, the identity automorphism of the ring $k$. For any $k$-algebra $R$ (with or without identity) it is clear how to give $R_+:=R\oplus k$ a structure of augmented $k$-algebra with identity. There is a short exact sequence $0\to R\to R_+\to k\to 0$ in the pointed category of all $k$-algebras.

\bigskip

Let us go back to rngs. On any rng $R$ it is possible to define the operation $x\circ y = x+y +xy$ \cite[p.~85]{FR}. Then $(R,\circ)$ is a multiplicative monoid with the additive zero $0_R$ of the ring $R$ as its identity. The Jacobson radical $J(R)$ of $R$ is the set of all $x\in R$ such that $yx$ is left invertible in  $(R,\circ)$ for all $y\in R$ (equivalently, the set of all $x\in R$ such that $xy$ is right invertible in  $(R,\circ)$ for all $y\in R$). The ideal $J(R)$ is contained in the group of units of the monoid $(R,\circ)$. For any rng $R$, one has that $J(R_+)=J(R)\oplus 0$.

Clearly, the mapping $\psi\colon S\to S$, $\psi\colon x\mapsto 1+x$, is a monoid isomorphism of the monoid $(S,\circ)$ onto the monoid $(S,\cdot{})$ for every ring $S$. In particular, $\varphi(J(S))\subseteq \varphi(U(S,\circ))=U(S,\cdot{})$ for any ring $S$, and $\varphi(J(R))=\varphi(J(R_+))\subseteq \varphi(U(R_+,\circ))=U(R_+,\cdot{})$. (More precisely, $J(S)$ is the largest of all the ideals $I$ of a ring $S$ such that $\varphi(I)\subseteq U(S,{}\cdot {})$.) This will be useful in the next subsection, devoted to braces.

\subsection{Braces}
 Following Drinfeld \cite{23}, a {\sl set-theoretic solution
of the Yang-Baxter equation} is a pair $(X, r)$, where $X$ is a set and
$r \colon X\times X \to X \times X$ is a bijection such that
$$(r\times\id)(\id\times r)(r\times\id)=(\id\times r)(r\times\id)(\id\times r).$$
 
 \smallskip 
 
 Rump \cite{Rump} proved that there is a  connection between radical rings and set-theoretic solutions of the Yang-Baxter equation. 
A {\sl radical ring} is a rng $R$ for which $R=J(R)$. Hence a rng $R$ is a radical ring if and only if the monoid $(R,\circ)$ is a group, and for such a ring $R$, the mapping
$$r \colon R\times R \to R \times R,\quad r(x,y) = (xy + y, (xy + y)'\circ x\circ y),$$
where $z'$ denotes the inverse of $z$ with respect to the operation $\circ$, is a non-
degenerate solution of the Yang-Baxter equation such that $r^2 = \id_{R\times R}$.

Rump \cite{Rump} generalized this to braces, and braces were generalized in 2017 to skew braces by L.~Guarnieri and L.~Vendramin \cite{GV}.
A {\sl (left) skew brace} is
a triple $(A, \circ, *)$, where $(A, \circ) $ and $(A,*)$ are groups (not necessarily abelian) such that $$a\circ (b * c) = (a\circ b)*a^{-1}* ( a\circ c)$$
for every $a,b,c\in A$. Here $a^{-1}$ denotes the inverse of $a$ in the group $(A,*)$. The inverse of $a$ in the group $(A,\circ)$ will be denoted by $a'$. Notice the similarity between this definition and the definition of a left near-ring with identity $(N,*,\cdot)$, where $(N,*)$ is a  group (not necessarily abelian), $(N,\cdot)$ is a monoid, and $a (b * c) = a b * a c$. From the next Lemma we will see the close relation between left skew braces and $G$-groups.

  \bigskip
 
 For every skew brace $(A, \circ, *)$, the mapping $$r \colon A\times A \to A \times A,\quad r(x,y) = (x^{-1} *(x\circ y),(x^{-1} *(x\circ y))'\circ x\circ y),$$
is a non-degenerate set-theoretic solution of the Yang-Baxter equation (\cite[Theorem~3.1]{GV} and \cite[p.~96]{KSV}).
 
\begin{lemma}\label{Bachi}{\rm \cite{B18}} Let $A$ be a skew brace. Then $\lambda\colon (A,\circ)\to\Aut (A,*)$, given by $\lambda\colon a\mapsto\lambda_a$, where $\lambda_a(b)=a^{-1}* (a\circ b)$, is a well-defined group homomorphism.\end{lemma}

  In other words, for any skew brace $(A, \circ, *)$, we have that $(A,*)$ is an $(A,\circ)$-group with respect to the action $\lambda\colon (A,\circ)\to\Aut (A,*)$ described in the statement of Lemma~\ref{Bachi}. Conversely, suppose that a set $A$ has two group structures 
   $(A, \circ) $ and $(A,*)$ and that $(A,*)$ is an $(A,\circ)$-group with respect to the action $\lambda\colon (A,\circ)\to\Aut (A,*)$, defined by $\lambda\colon a\mapsto\lambda_a$, where $\lambda_a(b)=a^{-1} *(a\circ b)$. Then the fact that each $\lambda_a$ is an automorphism yields that $\lambda_a(b*c)=\lambda_a(b)*\lambda_a(c)$, i.e., $ a^{-1}* (a\circ (b*c))=a^{-1}* (a\circ b)*a^{-1}* (a\circ c)$, from which $ a\circ (b*c)=(a\circ b)*a^{-1}* (a\circ c)$. Hence braces are exactly those particular $G$-groups $(H,\lambda)$ for which $G=H$ as sets, and the action $ \lambda$ is defined by $\lambda_a(b)=a^{-1}* (a\circ b)$.

  The semidirect product corresponding to such an $(A,\circ)$-group  $(A,*)$ is the group $P:=(A,*)\ltimes(A,\circ)$, i.e., the cartesian product $P:=A\times A$ with group operation defined by \begin{equation}(a_1,a_2)\cdot(b_1,b_2)=(a_1*a_2^{-1}* (a_2\circ b_1),a_2\circ b_2).\label{vhi}\end{equation} Conversely, given two group structures $(A, \circ) $ and $(A,*)$ on the same set $A$ such that $P:=A\times A$ with the operation as in (\ref{vhi}) is a group, then  $(A, \circ, *)$ is a left skew brace. This proves that left skew braces are particular groups: they are the groups $G$ for which there is a normal subgroup $A\unlhd G$ and a bijection $\alpha\colon G/A\to A$ with suitable properties. (The two group structures on $A$ are the original one as a subgroup of $G$ and the one brought from $G/A$ via the bijective $1$-cocycle $\alpha$.) 
  
  The original definition of brace given by Rump was that with  $(A,*)$ an additive abelian group. Henca a brace is an abelian group $(A,+)$ that is also a group $(A,\circ)$ with a monoid morphism $\lambda'\colon (A,\circ)\to\End(A)\times\Z$ whose image is contained in $\End(A)\times\{1\}$. Notice the similarity with the definition of $k$-algebra: Given a {\em commutative} ring $k$, a $k$-algebra is a ring $A$ with a ring morphism $\lambda\colon k\to A$ whose image is contained in the center $A$. Viewing a left skew brace $(A,\circ,*)$ as the corresponding semidirect product $P:=(A,*)\ltimes(A,\circ)$ with the group operation (\ref{vhi})  suggests us that the correct notion of action of a left skew brace $(A,\circ,*)$ over a group $H$ is that of $P$-group $H$. Cf.~\cite{RumpModules}. 
  
  Notice that the group $U(E\times\Z)$ of units of any ring $E\times\Z$ with identity and with an augmentation $\Z\to E\times\Z\to\Z$ has a subgroup $U(E\times\Z)\cap (E\times\{1\})$ of index 2. When $E$ already has an identity, then there is a unique morphism $u\colon E\times\Z\to E$ of rings with identity that extends the identity morphism $E\to E$. Thus, for a brace $A$,  one has the composite monoid morphism $\lambda:=u\circ\lambda'\colon A\to\End(A)$ (cf.~Lemma~\ref{Bachi}).

For a skew brace $(A, \circ,*)$, and more generally for any group $(A,*)$, we don't have a ring $\End(A,+)$ anymore, but a nearring $M(A,*)$, consisting of all mappings $A\to A$. Addition is defined by $(f+g)(a)=f(a)*g(a)$, multiplication is composition of mappings, the zero element is the zero map, i.e., the mapping which takes every element of $A$ to the identity element of $(A,*)$, the additive inverse $-f $ of $f \in M(G)$ is $ (-f)(a) = (f(a))^{-1}$ (inverse of $f(a)$ in the group $(A,*)$) for all $a\in A$, the identity of the nearring is the identity mapping $\iota_A\colon A\to A$. For the nearring $M(A,*)$ it is also possible to construct its Dorroh extension as in \cite{Veldsman}, because $M(A,*)$ trivially satisfies condition (BC): the Dorroh extension is the cartesian product $M(A,*)\times\Z$ with addition and multiplication defined by $$(f,a)+(g,b)=(f+a\iota_A +g-a\iota_A ,a+b)$$ $$(f,a)(g,b)=((f+a\iota_A )(g+b\iota_A )-(ab)\iota_A , ab).$$ As in \cite{Veldsman}, it is easy to verify that $M(A,*)\times\Z$ is a nearring with additive identity $(0,0)$, multiplicative identity $(0,1)$, and the additive inverse of $(f,a)$ is  $(-a\iota_A-f+a\iota_A, -a)$. Then $\Z\to M(A,*)\times\Z\to\Z$ is an augmented nearring with identity. The kernel of the augmentation $M(A,*)\times\Z\to\Z$, $(f,a)\mapsto a$, is the two-sided ideal $\{\,(f,0)\mid f\in M(A,*)\,\}$ of $M(A,*)\times\Z$, and is a nearring isomorphic to $M(A,*)$. There is a surjective nearring morphism $M(A,*)\times\Z\to M(A,*)$, $(f,a)\mapsto f+a\iota$. We will continue in this direction in a next paper, now we go back to the presentation of a skew brace as a semidirect product $P$ of groups, in order to introduce left skew brace morphisms and ideals of braces.

 \bigskip
 
 A {\em morphism of skew braces} is a mapping $f\colon A\to A'$ such that $f(a\circ b)=f(a)\circ f(b)$ and $f(a*b)=f(a)* f(b)$ for every $a,b\in A$. Every such morphism $f\colon A\to A'$ induces a corresponding group morphism $f\times f\colon P=A\ltimes A\to P'=A'\ltimes A'$, $(f\times f)(a,b)=(f(a), f(b))$. Since $(A, \circ)$ and $(A, *)$ have the same identity, the group morphisms $f\colon (A, \circ)\to (A, \circ)$ and $f\colon (A, *)\to (A, *)$ have the same kernel $I$, which is a normal subgroup both in $ (A, \circ)$ and in $ (A, *)$. The kernel of $f\times f\colon P=A\ltimes A\to P'=A'\ltimes A'$ is $I\times I$. The assignements $(A, \circ,*)\mapsto P:=(A,*)\ltimes(A,\circ)$ and $f\mapsto f\times f$ define a faithful functor of the category of left skew braces into the category of groups. This allows us to view left skew braces as groups and suggests us the definition of kernel of a left skew brace morphism  $f\colon A\to A'$ and of factor skew brace, getting a left skew brace isomorphism $(A/I,*)\ltimes(A/I,\circ)\to (f\times f)(P)$, as follows.
The {\em kernel} of a left skew brace morphism  $f\colon A\to A'$ is the inverse image $f^{-1}(1_A)$ of the identity $1_A$ of $ (A, \circ)$, (which coincides with the identity of  $ (A, *))$.

An {\em ideal} of a skew brace $A$ is a normal subgroup $I$ both of $(A,\circ)$ and $(A,*)$ such that $I\circ a=I*a$ for every $a\in I$. Notice that this implies that the factor groups of $(A,\circ)$ and $(A,*)$ modulo $I$ coincide as sets. Also notice that  $ I\circ a=I*a$ for every $a\in A$ if and only if, for every $a,b\in A$, one has $a*b^{-1}\in I$ if and only if $a\circ b'\in I$. Equivalently, if and only if $(x*a)\circ a'\in I$ and $(x\circ a)*a^{-1}\in I$ for every $x\in I$ and every $a\in A$. The semidirect product corresponding to the quotient left skew brace $A/I$ is $(A/I,*)\ltimes(A/I,\circ)$.

\bigskip

As far as the representation of a left skew brace as a $G$-group is concerned, i.e., as a $(A,\circ)$-group $(A,*)$, notice that if $I$ is an ideal of the left skew brace $A$, then $I$ is a normal $(A,\circ)$-subgroup of $(A,*)$, so that the group morphism  $\lambda\colon (A,\circ)\to\Aut (A,*)$ induces a group morphism  $(A,\circ)\to\Aut (A/I,*)$. But now $I$ is contained in the kernel of this morphism $(A,\circ)\to\Aut (A/I,*)$, which therefore induces a group morphism $\overline{\lambda}\colon (A/I,\circ)\to\Aut (A/I,*)$.

For any fixed left skew brace $A$, there is a canonical one-to-one correspondence between the set of all ideals of the left skew brace $A$ and the set of all equivalence relations on the set $A$ compatible with both the operations $\circ$ and $*$ (i.e., the equivalence relations $\sim$ on the set $A$ such that, for every $a,b,c,d\in A$, $a\sim b$ and $c\sim d$ imply $a\circ c\sim b\circ d$ and $a*c\sim b*d$. Given any ideal $I$ of the  left skew brace $A$, the corresponding equivalence relation on $A$ is ${\sim}_I$, defined, for every $a,b\in A$, by $a\sim_I b$ if $a^{-1}*b\in I$ (equivalently, if $a'\circ b\in I$; notice that the equivalence classes of $A$ modulo $\sim_I$ are the cosets $a*I=I*a=a\circ I=I\circ a$). Conversely, if $\sim$ is an equivalence relation on  $A$ compatible with both the operations $\circ$ and $*$, then the corresponding ideal of $A$ is the equivalence class of $1_A$ modulo $\sim$.

    \bigskip
   
Any intersection of ideals is an ideal, so that we get a complete lattice $\Cal I(A)$ of ideals for any skew brace $A$. In particular, every subset $X$ of $A$ generates an ideal, the intersection of the ideals that contain $X$. It can be ``constructively'' described as follows. Given a subset $X$ of a skew brace $A$, consider the increasing sequence $X_n$, $n\ge 0$, of subsets of $A$ where $X_0:=X$, $X_{n+1}$ is the normal closure of $X_n$ in $(A,*)$ if $n\equiv 0 \ (\mathrm{mod}\ 3)$, $X_{n+1}$ is the normal closure of $X_n$ in $(A,\circ)$ if $n\equiv 1 \ (\mathrm{mod}\ 3$, $X_{n+1}:=\bigcup_{a\in A}\lambda_a(X_n)$ if $n\equiv 2 \ (\mathrm{mod}\ 3)$. The ideal of $A$ generated by $X$ is $\bigcup_{n\ge0}X_n$.
     
   \bigskip 
   
  As multiplication in the complete lattice of all ideals of a left skew brace $A$ define as product of two ideals $I$ and $J$ of $A$ the ideal of the skew brace $A$ generated by $\pi_1([I\times I, J\times J]_P)\cup \pi_2([I\times I, J\times J]_P)$, where $[I\times I, J\times J]_P$ is the commutator in $P$ of its normal subgroups $I\times I$ and $ J\times J$, and the mappings $\pi_i\colon A\times A\to A$, $i=1,2$, are the two canonical projections.
  %$[I,J]_+\cup [I,J]_\circ$, where $[I,J]_+$ is the commutator of the normal subgroups $I$ and $J$ of $(A,+)$ in the group $(A,+)$, and  $[I,J]_\circ$  is the commutator of the normal subgroups $I$ and $J$ of $(A,\circ)$ in the group $(A,\circ)$. 
   With this operation,  $\Cal I(A)$ becomes a multiplicative lattice. (Notice: these notations and notions differ from similar previous notations and notions in the literature.) For instance, in our terminology a left skew brace is {\em abelian} if and only if the two operations $\circ$ and $*$ coincide and are commutative.  Similarly, the {\em socle} $$(4)\qquad \soc(A):=\{ \, a\in A\mid ba=0\ \mbox{\rm for every }b\in A \,\}$$ of a brace $A$ as defined in \cite{Rump} is the {\em right center} in our terminology. This has been independently observed by Vsevolod Gubarev.)
        
     The notions of prime ideal, solvable,  nilpotent, etc., skew braces now become natural. This must be compared with previous notions by Rump (left and right nilpotent braces and radical chains \cite{Rump}), Rowen (2017, see the first paragraph of Section~4 in \cite{KSV}), in \cite{BCJO19}, and \cite{KSV}, and will be described in detail in a subsequent paper.
     But notice that most of these previous notions differ from ours. For instance, the product of two ideals of a left brace $A$ in \cite{Rump} is not an ideal of the left brace, but only a subgroup of $(A,*)$, while in our theory the product of two ideals is always necessarily an ideal. Also, the ideal generated by the product of two ideals in the sense of Rump \cite{Rump} can be different from the product of the two ideals in our sense. For instance, in \cite[Example~2]{Rump} one has that $A^{(3)}=\{0,111\}$ is not an ideal, while in our notation $A^{(3)}=A^2=A^{(4)}$ is always necessarily an ideal, and $A$ is not right nilpotent and not left nilpotent in our sense, but is left nilpotent in the sense of Rump because $A^{(4)}=0$.

\bigskip

{\bf  Acknowledgements} This paper is the result of several conversations and e-mail messages I have exchanged with Professor George Janelidze for the past two years, in particular during the writing of the paper \cite{FFJ}. I am extremely grateful to him. I have learnt a lot of things from him.

%    Bibliographies can be prepared with BibTeX using amsplain,
%    amsalpha, or (for "historical" overviews) natbib style.
\bibliographystyle{amsplain}

\end{document}